\documentclass[12pt,reqno]{amsart}
\usepackage{latexsym,amsmath,mathtools}
\usepackage{amsfonts}
\usepackage{amssymb}
\usepackage{latexsym}
\usepackage{enumerate}
\usepackage[T1]{fontenc}
\usepackage{fourier}
\usepackage{bbm}
\usepackage{color}
\usepackage{xcolor}
\usepackage{tikz}
\usepackage{fullpage}
\usepackage[latin1]{inputenc}
\usepackage{anysize,cite}
\usepackage{graphicx}
\usepackage{url}
\usepackage{comment}
\usepackage{amssymb,amsmath,amscd,amsfonts,amsthm,mathrsfs,color}
\newcommand{\bb}[1]{\mathbb{#1}}

\theoremstyle{definition}
\numberwithin{equation}{section}
\newtheorem{theorem}{Theorem}
\newtheorem{question}[theorem]{Question}

\newtheorem{corollary}[theorem]{Corollary}
\newtheorem{lemma}[theorem]{Lemma}
\newtheorem{claim}[theorem]{Claim}
\newtheorem*{definition}{Definition}

\begin{document}
	\title{Large complete minors in random subgraphs}
	\date{}

	\author{Joshua Erde$^{*}$, Mihyun Kang$^{*}$, and Michael Krivelevich$^{\ddagger}$ \\ \\
		\today}
	\thanks{$^{*}$ 
		Institute of Discrete Mathematics, 
		Graz University of Technology, 
		Steyrergasse 30,
		8010 Graz,
		Austria,  
		{\tt \{erde,kang\}@math.tugraz.at}.
		Supported by Austrian Science Fund (FWF): I3747\phantom{}}
	\thanks{$^{\ddagger}$ 
		School of Mathematical Sciences, 
		Sackler Faculty of Exact Sciences, 
		Tel Aviv University,  
		Tel Aviv 6997801,
		Israel, 
		{\tt  krivelev@tauex.tau.ac.il}.
		Supported in part by USA-Israel BSF grant 2018267, and by ISF grant 1261/17.}
	
	\begin{abstract}
		Let $G$ be a graph of minimum degree at least $k$ and let $G_p$ be the random subgraph of $G$ obtained by keeping each edge independently with probability $p$. We are interested in the size of the largest complete minor that $G_p$ contains when $p = \frac{1+\varepsilon}{k}$ with $\varepsilon >0$. We show that with high probability $G_p$ contains a complete minor of order $\tilde{\Omega}(\sqrt{k})$, where the $\sim$ hides a polylogarithmic factor. Furthermore, in the case where the order of $G$ is also bounded above by a constant multiple of $k$, we show that this polylogarithmic term can be removed, giving a tight bound.
	\end{abstract}

	\maketitle

	\section{Introduction}
	The binomial random graph model $G(n,p)$, introduced by Gilbert \cite{G59}, is a random variable on the subgraphs of the complete graph $K_n$ whose distribution is given by including each edge in the subgraph independently with probability $p$. Since its introduction this model has been extensively studied. A particularly striking feature of this model is the `phase transition' that it undergoes at $p = \frac{1}{n}$, exhibiting vastly different behaviour when $p = \frac{1 - \varepsilon}{n}$ to when $p = \frac{1+\varepsilon}{n}$ (where $\varepsilon$ is a  positive constant). For more background on the theory of random graphs, see \cite{B01,FK16,JLR00}.
	
	More recently, the following generalisation of the binomial random graph model has attracted attention: Suppose $G$ is an arbitrary graph with minimum degree $\delta(G)$ at least $k-1$, and let $G_p$ denote the random subgraph of $G$ obtained by retaining each edge of $G$ independently with probability $p$. When $G= K_k$, the complete graph on $k$ vertices, we recover the binomial model $G(k,p)$.
	
	For several properties, it has been shown that once one passes the threshold for the occurence of the property which holds  in $G(k,p)$ with high probability\footnote{Here and throughout the paper, we will say that an event happens with high probability (whp) if the probability tends to one as $k \rightarrow \infty$. All asymptotics in the paper are taken as $k\rightarrow \infty$.} (as a function of $k$), or whp for short, these properties will also occur whp in $G_p$. For example, when $p= \frac{1+\varepsilon}{k}$ it has been shown that whp $G_p$ is non-planar \cite{FK13}, and contains a path or cycle of length linear in $k$ \cite{KS14,EJ18}. Similarly, when $p = \omega\left(\frac{1}{k}\right)$, whp $G_p$ contains a path or cycle of length $(1-o(1))k$ \cite{KLS15,R14} and when $p = (1+\varepsilon)\frac{\log k}{k}$, whp $G_p$ contains a path of length $k$ \cite{KLS15} and in fact even a cycle of length $k+1$ \cite{GNS17}. All of these results generalise known results about the binomial model.
	
	In this paper we will be interested in the size of the largest complete minor in a graph $G$, sometimes known as the \emph{Hadwiger number} of $G$, which we denote by $h(G)$. Fountoulakis, K{\"u}hn, and Osthus \cite{FKO08} showed the following bound for the Hadwiger number of $G(k,p)$ in the so-called \emph{supercritical regime}.
	\begin{theorem}[\cite{FKO08}]\label{t:G(n,p)}
		Let $\varepsilon$ be a positive constant and $p = \frac{1+ \varepsilon}{k}$. Then whp $h(G(k,p)) = \Theta\left( \sqrt{k}\right)$.
	\end{theorem}
	\noindent
	Using expanders Krivelevich \cite{K18} gave an alternative proof of the above theorem.
	
	As part of their work on the genus of random subgraphs, Frieze and Krivelevich \cite{FK13} noted that their proof actually shows that if $G$ is a graph with minimum degree at least $k$ and $p = \frac{1+\varepsilon}{k}$, then whp $h(G_p) = \omega(1)$, and asked what the largest function $t(k)$ is such that whp $h(G_p) \geq t(k)$.
	
	Our main result is a lower bound on $h(G_p)$, which is tight up to polylogarithmic factors.
	\begin{theorem}\label{t:general}
		Let $\varepsilon$ be a positive constant, $G$ be a graph with $\delta(G) \geq k$, and $p = \frac{1+\varepsilon}{k}$. Then whp 
		\begin{align*}
		h(G_p) \ = \  \Omega \left( \sqrt{\frac{k}{ \log k}}\right).
		\end{align*}
	\end{theorem}
	
	In other words, for any $\varepsilon >0$, there exists a constant $c=c(\varepsilon)$ and a function $f : \mathbb{N} \rightarrow [0,1]$ such that if $k \in \bb{N}$ is large enough, $\left(G^i \colon i \in \bb{N}\right)$ is a sequence of graphs with $\delta\left(G^i\right) \geq k$, and $p = \frac{1 + \varepsilon}{k}$, then
	\[
	\bb{P}\left( h\left(G^i_{p}\right) \leq c \sqrt{\frac{k}{ \log k}}\right) \leq f(k),
	\]
	and $f(k) \rightarrow 0$ as $k \rightarrow \infty$.
	
	Using ideas similar to the proof of Krivelevich in \cite{K18} we are able to remove the polylogarithmic factor, and to give the following asymptotically tight bound, when the number of vertices in $G$ is linear in $k$.
	
	\begin{theorem}\label{t:dense}
		Let $\nu$ and $\varepsilon$ be positive constants, $G$ be a graph on $n$ vertices with $\delta(G) \geq k \geq \nu n$, and $p=\frac{1 + \varepsilon}{k}$. Then whp 
		\begin{align*}
		h(G_p)   \ = \   \Omega\left(\sqrt{k}\right).
		\end{align*}
	\end{theorem}
	
	Note that if $k=\Theta(n)$, then whp the number of edges in $G_p$ is at most $(1+\varepsilon)\frac{n^2}{k} = O(n)$. Hence, since any graph with a $K_t$ minor must contain at least $e(K_t) = \binom{t}{2}$ many edges, it follows that whp $h(G_p) = O(\sqrt{n}) = O(\sqrt{k})$, and so this bound is indeed asymptotically tight. We would be interested to know if this is the correct bound for all ranges of $k$.
	
	\begin{question}
		Let $\varepsilon$ be a positive constant, $G$ be a graph with $\delta(G) \geq k$, and $p = \frac{1+\varepsilon}{k}$. Is $h(G_p) ={\Omega}(\sqrt{k})$ whp? 
	\end{question}
	
	A key ingredient in our proof will be the following lemma, which roughly says that if we have a forest $T$ of order $n$ whose components are all of size around $\sqrt{k}$ and a set $F$ of $\Theta(kn)$ many edges on the same vertex set as $T$, and if $p = \Theta \left(\frac{1}{k} \right)$, then whp the random subgraph $T \cup F_p$ will contain a complete minor of order around $\sqrt{k}$. 
	
	\begin{lemma}\label{l:sprinklingminor}
		Let $k = \omega(1)$ and $n = \omega\left(\sqrt{k}\right)$ be integers, and $b_1,c_1,c_2 > 0$ and $b_2 >1$ be constants. Suppose $V$ is a set of $n$ vertices, $T$ is a spanning forest of $V$ with components $A_1 \ldots, A_r \subseteq V$ such that $b_1 \sqrt{k} \leq |A_i| \leq b_2 \sqrt{k}$, $F$ is a set of $c_1 k n$ many edges on the vertex set $V$, and $p = \frac{c_2}{k}$. Then whp 
		\[
		h(T \cup F_p) = \Omega\left(\sqrt{\frac{k}{\log k}}\right).
		\]
	\end{lemma}

	The paper is structured as follows: In Section \ref{s:prelim} we will introduce the relevant background material and some useful lemmas. In Section \ref{s:lemma} we will give a proof of Lemma \ref{l:sprinklingminor} and then in Sections \ref{s:general} and \ref{s:dense} we will give proofs of Theorems \ref{t:general} and \ref{t:dense}.
	
	\subsection*{Notation} We will throughout the paper omit floor and ceiling signs to simplify the presentation. We will write $\log$ for the natural logarithm and given a graph $G$ we denote by $|G|$ the number of vertices in $G$.

	\bigskip
	\section{Preliminaries}\label{s:prelim}
	We will use the following bound, originally from Kostochka \cite{K82,K84} and Thomason \cite{T84}, which says a graph of large average degree contains a large complete minor.
	\begin{lemma}[\cite{T01}]\label{l:minor}
		If the average degree of $G$ is at least $t \sqrt{\log{t}}$, then $h(G) \geq t$.
	\end{lemma}
	
	\begin{corollary}\label{c:usefulminor}
		If the average degree of $G$ is at least $t$, then $h(G) = \Omega\left( \frac{t}{\sqrt{\log t}}\right)$.
	\end{corollary}
	
	We will also want to use the following simple lemma, which essentially appears in \cite{KN06}, to decompose a tree into roughly equal sized parts.
	
	\begin{lemma}[{\cite[Proposition 4.5]{KN06}}]\label{l:decomp}
		Let $T$ be a rooted tree on $n$ vertices with maximum degree $\Delta$, and let $1\leq \ell\leq n$ be an integer. Then there exists a vertex $v\in V(T)$ such that the subtree $T_v$ of $T$ rooted at $v$ satisfies  $\ell \leq |T_v| \leq \ell \Delta$.
	\end{lemma}
	
	As a corollary we have the following decomposition result for a tree with bounded maximum degree.
	\begin{corollary}\label{c:treedecomp}
		If $T$ is a tree with $\Delta(T) \leq C$ and $|T| > \sqrt{k}$, then there exist disjoint vertex sets  $A_1 \ldots, A_r \subseteq V(T)$ such that
		\begin{itemize}
			\item $V(T) = \bigcup_{i=1}^r A_i$; 
			\item $T[A_i]$ is connected for each $i$; and 
			\item $\sqrt{k} \leq |A_i| \leq  (C+1) \sqrt{k}$ for each $i$.
		\end{itemize}
	\end{corollary}
	
	We will need the following simple bound on the expectation of a restricted binomial random variable.
	\begin{lemma}\label{l:restricted}
		Let $X \sim \text{Bin}(n,p)$ be a binomial random variable with $2enp < K$ for some constant $K>0$. If $Y = \min\{ X , K\}$, then
		\[
		\bb{E}(Y) \geq np - K2^{-K}.
		\]
	\end{lemma}
	\begin{proof}
		For every $t \leq K$ we have that $\bb{P}(Y=t) \geq \bb{P}(X=t)$. Hence, by standard estimates
		
		\begin{align*}
		\bb{E}(X) -\bb{E}(Y) &\leq \sum_{t > K} t \binom{n}{t} p^t (1-p)^{n-t}\\
		&\leq \sum_{t > K} t \left(\frac{enp}{t}\right)^t\\
		&\leq \sum_{t > K} enp \left(\frac{enp}{t}\right)^{t-1}\\
		&\leq \sum_{t > K} \frac{K}{2} \left(\frac{enp}{K}\right)^{t-1}\\
		&\leq \frac{K}{2} \left(\frac{enp}{K}\right)^{K-1}\\
		&\leq K2^{-K},
		\end{align*}
		since $\frac{enp}{K} < \frac{1}{2}$.
	\end{proof}
	
	We will use the following generalised Chernoff type bound, due to Hoeffding.
	
	\begin{lemma}[\cite{H63}]\label{l:hoeffding}
		Let $K>0$ be a constant and let $X_1,\ldots, X_n$ be independent random variables such that $0\leq X_i \leq K$ for each $i \leq n$. If $X = \sum_{i=1}^n X_i$ and $t\geq 0$ then
		\[
		\bb{P}\left(|X-\bb{E}(X)| \geq t\right) \leq 2 \text{exp}\left(-\frac{t^2}{nK^2}\right).
		\]
	\end{lemma}

	\smallskip
	\section{Large complete minors: proof of Lemma \ref{l:sprinklingminor}}\label{s:lemma}
	
	Since $V = \bigcup_{i=1}^r A_i$ and $b_1 k^{\frac{1}{2}} \leq |A_i| \leq b_2 k^{\frac{1}{2}}$, it follows that $r \leq \frac{1}{b_1}k^{-\frac{1}{2}}n$. Let $F'$ be the set of edges in $F$ which are not contained in any $A_i$. Then, since each $A_i$ contains at most $\binom{|A_i|}{2} \leq \frac{b_2^2}{2}k$ edges inside it and  $|F|\geq  c_1kn $, it follows that for large $k$,
	\[
	|F'| \geq |F| -  r\frac{b_2^2}{2}k \geq  c_1kn - \frac{b_2^2}{2b_1} \sqrt{k}n \geq \frac{c_1}{2}kn.
	\]
	Hence, on average each $A_i$ meets at least $\frac{2|F'|}{r} \geq c_1 b_1 k^{\frac{3}{2}}$ many edges in $F'$. 
	We recursively delete sets $A_i$, and the edges in $F'$ incident to them, which meet at most $\frac{c_1 b_1}{4}k^{\frac{3}{2}}$ edges remaining in $F'$;  we must eventually stop this process before exhausting the $A_i$,
	since $r \leq \frac{1}{b_1}k^{-\frac{1}{2}}n$ (i.e. there are at most $\frac{1}{b_1}k^{-\frac{1}{2}}n$ many $A_i$) and 
	\[
	\frac{c_1 b_1 }{4}k^{\frac{3}{2}} \frac{1}{b_1}k^{-\frac{1}{2}}n = \frac{c_1}{4}kn \leq \frac{|F'|}{2}.
	\]
	Hence there is some subfamily, without loss of generality, $\{A_1, \ldots, A_\ell\}$ of the $A_i$, and some subset $F'' \subseteq F'$ of edges which lie between $A_i$ and $A_j$ with $i,j \in [\ell]$ such that at least $\frac{c_1 b_1 }{4}k^{\frac{3}{2}}$ edges of $F''$ meet each $A_i$.
	
	Note that $0 \leq e_{F''}(A_i,A_j) \leq b_2^2 k$ for each pair $i,j \in [\ell]$. For each pair $i,j \in [\ell]$ such that $e_{F''}(A_i,A_j) >k$ let us delete $e_{F''} (A_i,A_j)- k$ many edges in $F''$ which lie between $A_i$ and $A_j$, and call the resulting set of edges $\hat{F}$. Then $0 \leq e_{\hat{F}}(A_i,A_j) \leq k$ for each $i,j \in [\ell]$ and furthermore each $A_i$ still meets at least $\frac{c_1b_1}{4 b_2^2}k^{\frac{3}{2}}$ many edges of $\hat{F}$.
    Indeed, the proportion of the edges in $F''$ between each pair $A_i$ and $A_j$ that we delete is at most $\left(1- \frac{1}{b_2^2}\right)$, and hence at least a $\frac{1}{b_2^2}$ proportion of the edges meeting each $A_i$ remains.
	 In particular we have
	\begin{equation}\label{e:edges}
	\sum_{i,j \in [\ell]}  e_{\hat{F}}(A_i,A_j)   \ \geq \  \ell  \frac{c_1b_1}{2 b_2^2} k^{\frac{3}{2}}.
	\end{equation} 
	
	Let $H$ be an auxilliary (random) graph on $[\ell]$ such that $i \sim j$ if and only if there is an edge between $A_i$ and $A_j$ in $\hat{F}_p$. The number of edges between $A_i$ and $A_j$ in $\hat F_p$ is distributed as Bin$(e_{\hat{F}}(A_i,A_j),p)$. Note that if $mp < 1/2$, then $\bb{P}\left(\text{Bin}(m,p) \neq 0 \right) = 1- (1-p)^m \geq \frac{mp}{2}$. Since $e_{\hat{F}}(A_i,A_j) \leq k$ and $p=\frac{c_2}{k}$, and without loss of generality we may assume that $c_2 < \frac{1}{2}$, it follows that
	\begin{equation}\label{e:prob}
	\bb{P}(i \sim j) \geq \frac{c_2 e_{\hat{F}}(A_i,A_j)}{2k}.
	\end{equation}
	
	By \eqref{e:edges} and \eqref{e:prob}, we have
	\[
	\bb{E}(e(H)) = \frac{1}{2} \sum_{i,j \in [\ell]}\bb{P}(i \sim j) 
	\geq  \frac{1}{2} \sum_{i,j \in [\ell]} \frac{c_2 e_{\hat{F}}(A_i,A_j)}{2k} 
	\geq \frac{1}{4k} \ell \frac{c_1 c_2 b_1}{2 b_2^2}k^{\frac{3}{2}}  
	= \frac{c_1 c_2 b_1}{8 b_2^2}\ell k^{\frac{1}{2}}.
	\]
	Summing up, we have $v(H) = \ell$ and $\bb{E}(e(H)) = \Omega\left(\ell k^{\frac{1}{2}}\right)$, and so we expect $H$ to have average degree $\Omega\left(k^{\frac{1}{2}}\right)$. It remains to show that $e(H)$ is well concentrated about its mean $\mu := \bb{E}(e(H))$.

	Since $e(H)$ can be expressed as the sum of independent indicator random variables, a standard calculation shows that Var$(e(H)) \leq \mu$ and so, by Chebyshev's inequality,
	\[
	\bb{P}\left( |e(H) - \mu| \geq \mu^{\frac{2}{3}}\right) \leq \frac{\text{Var}(e(H))}{\mu^{\frac{4}{3}}}\leq \mu^{-\frac{1}{3}}= o(1). 
	\]

	Hence, whp $e(H) \geq (1 - o(1)) \mu$ and so whp $H$ has average degree $\Omega\left(k^{\frac{1}{2}}\right)$. Thus, by Corollary~\ref{c:usefulminor} whp
	\[
	h(H)=\Omega\left(\sqrt{\frac{k}{ \log k}}\right).
	\]
	Observe that by contracting each $A_i$ the graph $H$ becomes a minor of $T \cup F_p$, and so the result follows.

	\smallskip
	\section{The general case: Proof of Theorem \ref{t:general}}\label{s:general}
	We will broadly follow the strategy of Frieze and Krivelevich \cite{FK13} and their proof that whp $G_p$ is non-planar when $\delta(G) \geq k$ and $p = \frac{1+\varepsilon}{k}$. Using a lemma similar to Lemma \ref{l:sprinklingminor} they showed that if there is a tree $T$ in $G_{p_1}$, where $p_1 = \frac{1+\frac{\varepsilon}{2}}{k}$, with small maximum degree and $\Omega(|T|k)$ many edges in $G$ then, after exposing these edges with probability $p_2 \geq \frac{\varepsilon}{2k}$, the resulting graph will whp be non-planar. Since by Corollary \ref{c:treedecomp} we can split such a tree into components of size around $\sqrt{k}$, we can use Lemma \ref{l:sprinklingminor} in a similar fashion to find a large complete minor in this case.
	
	In order to find such a tree, Frieze and Krivelevich first build a small tree $T_1$ with small maximum degree, and then in stages iteratively expose the edges leaving the frontier $S_t$ (i.e. the set of active leaves) of the current tree $T_t$ under the assumption that $|S_t| = \Theta( |T_t|)$ and that the maximum degree in $T_t$ is small (in their argument polylogarithmic in $k$).
	
	If many of the edges leaving $S_t$ go back into the tree $T_t$, then we can apply Lemma \ref{l:sprinklingminor} as above to find a large complete minor. Otherwise, many of the edges leave $T_t$, in which case Frieze and Krivelevich showed that one can either find a dense subgraph between $S_t$ and its neighbourhood, and so also a large complete minor by Theorem \ref{t:dense}, or add a new layer of significant size to the current tree, whilst keeping the maximum degree bounded, allowing one to grow a slightly larger tree. Since this process cannot continue indefinitely, as $G$ is finite, eventually the tree stops growing, and we find our large minor.
	
	However, one cannot guarantee that the dense subgraph one finds is particularly dense, and so following this strategy naively only produces a minor of size \emph{logarithmic} in $k$. Instead, by exposing (the edges emanating from) the vertices of $S_t$ \emph{sequentially}, we will show that if we cannot continue the tree growth then at some point during the process there are many edges in $G$ between the new layer of growth and the remaining vertices in $S_t$, allowing us to apply Lemma \ref{l:sprinklingminor} as before.

	\begin{proof}[Proof of Theorem \ref{t:general}]
		Our plan will be to sprinkle with $p_1 = \frac{1 + \frac{\varepsilon}{2}}{k}$ and $p_2=\frac{p-p_1}{1-p_1} \geq \frac{\varepsilon}{2k}$.
		
		\subsubsection*{Initial Phase :}
		We first run an initial phase in which we build a partial binary tree $T_0$ of size $\log \log \log k =: N$ or $N+1$ in $G_{p_1}$. By a partial binary tree we mean a rooted tree, rooted at a leaf $\rho$, in which all vertices have degree three or one, such that there is some integer $L$ such that every non-root leaf is at distance $L$ or $L-1$ from $\rho$.
		
		We will do so via a sequence of trials. In a general stage we will have a set of \emph{discarded vertices} $X$ which will have size $o\left(\log k\right)$, and a partial binary tree $T'$ of size $< N$, such that so far we have only exposed edges in $G_{p_1}$ which meet either $X$, the root of $T'$, or a non-leaf vertex of $T'$.

		If $T'$ is a single vertex, let $v$ be the root of $T'$, otherwise let $v \in V(T')$ be a non-root leaf of minimal distance to the root. We expose the edges between $v$ and $V \setminus (X \cup V(T')$ in $G_{p_1}$. If $v$ has at least two neighbours, we choose two of them arbitrarily and add them to $T'$ as children of $v$, choosing and adding only one if $v$ is the root of $T'$. Otherwise we say that the trial \emph{fails} and we add $V(T')$ to $X$ and choose a new root $v$ arbitrarily from $V \setminus X$ and set $T' = v$. If at any point $|T'|= N$ or $N+1$ we set $T_0 := T'$ and we finish the initial phase.
		
		Since each $v$ has at least $k - |X \cup V(T')| \geq (1-\varepsilon)k$ many neighbours in $V \setminus (X \cup V(T'))$, the probability that a trial fails is at most
		\begin{align*}
		\bb{P}\left( \text{Bin}\left((1-\varepsilon)k , p_1 \right) < 2 \right) &= (1-p_1)^{(1-\varepsilon)k} + (1-\varepsilon)kp_1(1-p_1)^{(1-\varepsilon)k-1}\\
		&\leq \left(1 - p_1 + (1-\varepsilon)\left(1+\frac{\varepsilon}{2}\right)\right)\text{exp}\left(-(1+\frac{\varepsilon}{2})(1-\varepsilon) + p_1\right)\\
		&\leq 2 e^{\varepsilon-1}=: 1-\gamma < 1.
		\end{align*}
		
		Since each successful trial, apart from the first, adds two new vertices to $T'$, each time we choose a new root the probability that we build a suitable $T_0$ before a trial fails is at least $\gamma^{N}$.
		
		Therefore, 
		whp we build such a tree before we've chosen $\gamma^{-N}N$ new roots. Since we only ever discard at most $N$ vertices, during this process the number of discarded vertices is at most
		\[
		\gamma^{-N} N^2 = (\log \log k)^{-\log \gamma} \left( \log \log \log k \right)^2 = o\left(\log k\right).
		\]
		
		Let $S_0$ be the set of non-root leaves of $T_0$. Since $T_0$ is a partial binary tree as defined above, $T_0$ is contained in a full binary tree of depth $L$ rooted at $\rho$, and so $|T_0| \leq 2^{L}$, and since all of its non-root leaves are at depth $L-1$ or $L$, it follows that $|S_0| \geq 2^{L-2}$. In particular, $|S_0| \geq \frac{1}{4}|T_0|$. Furthermore, during this process we have only exposed edges which are incident to either a vertex in $X$ or a vertex in $V(T_0) \setminus S_0$. In particular, we have not exposed any edges between $S_0$ and $V \setminus (X \cup V(T_0))$.
		
		\subsubsection*{Tree Branching Phase :}
		Suppose then that in a general step we have a tree $T_t$ together with a set $S_t$ of leaves of $T_t$, called the \emph{frontier} of $T_t$, with the following properties:
		\begin{enumerate}[(a)]
			\item\label{c:size} $|S_t| \geq \frac{\varepsilon}{16} |T_t|$;
			\item\label{c:frontier} No edges from $S_t$ to $V \setminus (X \cup V(T_t))$ have been exposed in $G_{p_1}$;
			\item\label{c:maxdeg} The maximum degree in $T_t$ is at most $K +1$,
		\end{enumerate}
		where 
		\[
		K := 4\log \frac{1}{\varepsilon}
		\]
		is a large constant. Note that $T_0$ and $S_0$ satisfy these three properties.
		
		Let $0<\delta \ll \varepsilon$ and let us consider the set
		\[
		V_0 = V_0(t) := \left\{ s \in S_t \colon e_G(s,T_t) \geq \delta k \right\}.
		\]
		If $|V_0| \geq \delta |S_t|$, then $G[V(T_t)]$ contains a set $F$ of at least $\frac{\delta^2}{2} |S_t|k \geq \frac{\delta^2 \varepsilon}{32}|T_t|k$ many edges. In particular, note that this implies that $|T_t| = \Omega(k)$.
		
		Since $T_t$ has bounded degree, by Corollary \ref{c:treedecomp} we can split it into connected pieces of size $\Theta(\sqrt{k})$, and hence by Lemma \ref{l:sprinklingminor} when we sprinkle onto the edges of $F$ with probability $p_2$, whp we obtain a complete minor of order $\Omega\left( \sqrt{\frac{k}{\log k}}\right)$.
		
		So, we may assume that $|V_0| \leq \delta |S_t|$. Let $V_1 =V_1(t):= S_t \setminus V_0$. Since $|X| = o(k)$, every vertex $s \in V_1$ has degree at least $(1- 2\delta)k$ to $V \setminus (X \cup V(T_t))$. Let us arbitrarily order the set $V_1 = \{s_1,\ldots, s_r\}$ where $r:=|V_1|$. 
		
		We will build the new frontier $S_{t+1}$ by exposing the neighbourhood of each $s_i$ in turn. At the start of the process each $s_i$ has at least $(1- 2\delta)k$ many possible neighbours, however, as $S_{t+1}$ grows, it may be that some $s_i$ have a significant fraction of their neighbours inside $S_{t+1}$.
		
		Let us initially set $S_{t+1}(0) = \emptyset$ and $B(0) = \emptyset$. We will show that whp we can either find a large complete minor, or construct, for each $1 \leq j \leq r$, sets $S_{t+1}(j)$ and $B(j)$, and a forest $F(j)$, such that:
		\begin{enumerate}
			\item\label{c:badvtx} $B(j) \subseteq \{s_i \colon i \in [j]\}$ and $|B(j)| < \delta |S_t|$; 
			\item\label{c:largedeg} Each $s \in B(j)$ has $e_G(s,S_{t+1}(j)) \geq \delta k$;
			\item\label{c:forest} There is a forest $F(j)$ of maximum degree $K$ in $G_{p_1}$, whose components are stars centred at vertices in $\{s_i \colon i \in [j]\}$, such that $F(j)$ contains every vertex of $S_{t+1}(j)$.
		\end{enumerate}
		Clearly this is satisfied with $j=0$. Suppose we have constructed appropriate $S_{t+1}(j-1)$ and $B(j-1)$. 
		
		If $d_G(s_j,S_{t+1}(j-1)) \geq \delta k$ then we let $B(j) = B(j-1) \cup s_j$, $S_{t+1}(j) = S_{t+1}(j-1)$ and $F(j) = F(j-1)$. If $|B(j)| \geq \delta |S_t|$ then we can apply Lemma \ref{l:sprinklingminor} to the edges spanned by $V(T_t \cup F(j))$, those include the edges in $E_G(B(j),S_{t+1}(j))$.
		
		By our assumptions $T_t \cup F(j)$ has bounded maximum degree, and so by Corollary \ref{c:treedecomp} we can split it into connected parts of size around $\sqrt{k}$. Furthermore, $|T_t \cup F(j)| \leq |T_t| + K|S_t| = \Theta(|T_t|)$ and
		\[
		\big|E\big( G[V(T_t \cup F(j))] \big)\big| \geq e_G(B(j),S_{t+1}(j)) \geq \delta^2 |S_t| k = \Theta(|T_t|k).
		\]
		Hence, by Lemma \ref{l:sprinklingminor} after sprinkling onto $G[V(T_t \cup F(j))]$ with probability $p_2$ whp we have a complete minor of order $\Omega\left( \sqrt{\frac{k}{\log k}}\right)$.
		
		Therefore, we may assume that $|B(j)| < \delta |S_t|$ and so conditions (\ref{c:badvtx})--(\ref{c:forest}) are satisfied by $B(j)$, $S_{t+1}(j)$ and $F(j)$.
		
		So, we may assume that $d_G(s_j,S_{t+1}(j-1)) \leq \delta k$, and hence $s_j$ has at least $(1- 3\delta)k$ neighbours in $V \setminus (V(T_t) \cup S_{t+1}(j-1))$. We expose the neighbourhood $N(j)$ of $s_j$ in $V \setminus (V(T_t) \cup S_{t+1}(j-1))$ in $G_{p_1}$. Let us choose an arbitrary subset $N'(j) \subseteq N(j)$ of size $\min \{ N(j),K\}$ and let $F'(j)$ be the set of edges from $s_j$ to $N'(j)$. We set $B(j) = B(j-1)$, $S_{t+1}(j) = S_{t+1}(j-1) \cup N'(j)$ and $F(j) = F(j-1) \cup F'(j)$. It is clear that these now satisfy (\ref{c:badvtx})--(\ref{c:forest}).
		
		Hence we may assume that we have constructed $S_{t+1}(r)$, $B(r)$, and $F(r)$. Let us set $S_{t+1} = S_{t+1}(r)$ and $T_{t+1} = T_t \cup F(r)$. Note that $S_{t+1}$ is the frontier of $T_{t+1}$, and so property (\ref{c:frontier}) is satisfied. Furthermore, since $F(r)$ has maximum degree $K$, property (\ref{c:maxdeg}) is satisfied. 
		
		Finally, we note that, since $|B(r)| < \delta |S_t|$, we exposed the neighbourhood $N(j)$ of at least $(1-2\delta)|S_t|$ of the vertices in $S_t$. Furthermore, the size of the union of their neighbourhoods stochastically dominates a sum of restricted binomial random variables. More precisely, if we let 
		\[
		Y \sim \min \left\{ \text{Bin}\left((1- 3\delta)k, p_1\right),K \right\},
		\]
		then the sizes of the neighbourhoods $(N'(i) \colon i \not\in B(r))$ stochastically dominate a sequence of $r - |B(r)|$ many mutually independent copies of $Y$, $(Y_i \colon i \not\in B(r))$. Hence, if we let $Z=\sum_{i \not\in B(r)} Y_i$ then $|S_{t+1}|$ stochastically dominates $Z$. 
		
		Note that
		\[
		1+ \frac{\varepsilon}{3} \leq (1- 3\delta )k p_1 = (1- 3\delta)\left(1 + \frac{\varepsilon}{2}\right) \leq 2.
		\]
		
		Hence, since $K = 4 \log\frac{1}{\varepsilon} \geq 2e(1- 3\delta)k p_1$, Lemma \ref{l:restricted} implies that
		\begin{align*}
		\bb{E}(Y) &\geq \left(1 + \frac{\varepsilon}{3}\right) - K2^{-K}\\
		&\geq \left(1 + \frac{\varepsilon}{3}\right)  - Ke^{-\frac{K}{2}}\\
		&= \left(1 + \frac{\varepsilon}{3}\right)  - 4 \log \left( \frac{1}{\varepsilon} \right)\varepsilon^2\\
		&\geq 1 + \frac{\varepsilon}{4},
		\end{align*}
		as long as $\varepsilon$ is sufficiently small.
		
		Since $r - |B(r)| \geq (1- 2\delta)|S_t|$, it follows that $\bb{E}(Z) \geq (1- 2\delta)|S_t|\bb{E}(Y) \geq (1+\frac{\varepsilon}{5})|S_t|$, and so by Lemma \ref{l:hoeffding} we have that
		\begin{align}
		\bb{P}\left(|S_{t+1}| < \left(1+ \frac{\varepsilon}{8}\right)|S_t|\right) &\leq \bb{P}\left(Z < \left(1+ \frac{\varepsilon}{8}\right)|S_t|\right) \nonumber\\
		&\leq \bb{P}\left(|Z - \bb{E}(Z)| >  \frac{\varepsilon}{20}|S_t|\right)\nonumber\\
		&\leq 2\text{exp}\left(- \frac{\varepsilon^2 |S_t|^2}{400(r - |B(r)|)K^2}\right) \nonumber\\
		&= e^{-\Omega(|S_t|)} \label{e:failure},
		\end{align}
		since $r \leq |S_t|$. It follows that with probability at least $1- e^{-\Omega(|S_t|)}$,  $|S_{t+1}| \geq (1+ \frac{\varepsilon}{8})|S_t|$, and it is then a simple check that $|S_{t+1}| \geq \frac{\varepsilon}{16}|T_{t+1}|$ and hence property (\ref{c:size}) is also satisfied.
		
		Hence, we have shown that in the $t$th step we can either find a large complete minor, or with probability at least $1- e^{-\Omega(|S_t|)}$ we can continue our tree growth. However, since $G$ is finite the tree growth cannot continue forever, and so, unless the tree growth fails at some step, we must eventually find a large minor. 
		
		Recall that the probability of failure is $o(1)$ in the initial phase, and by (\ref{e:failure}) the probability that the tree growth fails at some step is at most
		\[
		\sum_{t} e^{-\Omega(|S_t|)} = o(1),
		\]
		since $|S_0| \geq \frac{1}{4} \log \log \log k$ and $|S_t| \geq (1+\frac{\varepsilon}{8}) |S_{t-1}|$. Hence the total probability of failure is $o(1)$, and so whp $G_p$ contains a large minor.
	\end{proof}
	
	\smallskip

	\section{The dense case: proof of Theorem \ref{t:dense}}\label{s:dense}
	We will need some auxilliary concepts and results to prove Theorem \ref{t:dense}.
	
	\begin{definition}
		Let $\alpha > 0$ be given. A graph $G$ on $n$ vertices is  an \emph{$\alpha$-expander}  if for every set of vertices $U \subseteq V(G)$ with $|U| \leq \frac{n}{2}$ the external neighbourhood of $U$, denoted by $N_G(U)$, satisfies
		\[
		|N_G(U)| \geq \alpha |U|.
		\]
	\end{definition}
	
	The following is given as a corollary of Theorem 8.4 in \cite{K19}.
	
	\begin{lemma}[\cite{K19}]\label{t:expandingminor}
		If $G$ is an $\alpha$-expander on $n$ vertices with bounded maximum degree, then $h(G) = \Omega( \sqrt{n})$.
	\end{lemma}
	
	We note that it follows from results announced in \cite{KR10} that the conclusion holds without the bounded maximum degree assumption.
	
	\begin{definition}
		Let $c_1 > c_2 > 1$ and let $\beta > 0$. A graph $G$ is  \emph{$(c_1,c_2,\beta)$-locally sparse} if
		\begin{itemize}
			\item $e(G) \geq c_1 |G|$; and 
			\item for every $U \subseteq V(G)$ such that $|U|\leq \beta |G|$ we have $e_G(U) \leq c_2 |U|$.
		\end{itemize}
	\end{definition}
	
	\begin{lemma}[{\cite[Theorem 1.1]{K18}}]\label{t:locallysparseexpander}
		Let $G$ be a $(c_1,c_2,\beta)$-locally sparse graph on $n$ vertices with maximum degree $\Delta$. Then $G$ contains an induced subgraph on $\beta n$ vertices which is a $\gamma$-expander for some positive $\gamma = \gamma(c_1,c_2,\beta,\Delta)$.
	\end{lemma}
	
	\smallskip
	\begin{proof}[Proof of Theorem \ref{t:dense}]
		Let $p_1 = \frac{1+\varepsilon/2}{k}$ and $p_2 =\frac{ p -p_1}{1-p_1} \geq \frac{\varepsilon}{2k}$. We will first give a series of claims about typical properties of $G_{p_1}$, which together with  Lemmas \ref{t:expandingminor} and \ref{t:locallysparseexpander}  will imply the theorem, and then give proofs of the claims.
		
		Firstly, we claim that there exists a constant $c_1 > 0$ such that whp there is some component $C_0$ of $G_{p_1}$ with at least $c_1 k$ vertices.
		\begin{claim}[{\cite[Theorem 4]{KS13}}]\label{c:linearcomponent}
			Whp $G_{p_1}$ contains a connected component $C_0$ with at least $\frac{\varepsilon^2}{5} k$ vertices.
		\end{claim}
		
		Next we claim that whp every large component in $G_{p_1}$  spans many edges in $G$.
		
		\begin{claim}\label{c:edgesinlargecomp}
			There exists a constant $c_2= c_2(c_1,\varepsilon,\nu)>0$ such whp for every connected component $C$ of $G_{p_1}$ of order at least $c_1 k$ we have $e_G(C) \geq c_2 k |C|$.
		\end{claim}
		
		As a consequence of Claim \ref{c:edgesinlargecomp}, whp the component $C_0$ with at least $\frac{\varepsilon^2}{5} k$ vertices (from Claim \ref{c:linearcomponent}) spans many edges in $G$. More precisely, we have that  whp $e_G(C_0) \geq c_2 k |C_0|$ and so, by the Chernoff bound, whp after we sprinkle with probability $p_2 \geq \frac{\varepsilon}{2k}$ into $C_0$, we have 
		\begin{equation}\label{e:excess}
		e(G_p[C_0]) \geq |C_0| + \frac{c_2 \varepsilon |C_0| }{4} \geq \left(1 + \frac{c_2 \varepsilon}{4}\right) |C_0| =: c_3 |C_0|.
		\end{equation}
		Let $c_4 := 1 + \frac{c_2 \varepsilon}{8}$, note that $c_3 > c_4 > 1$.
		
		\begin{claim}\label{c:fewedges}
			There exists a constant $\beta=\beta(c_4, \varepsilon,\nu)>0$ such that whp for every $U \subseteq V(G)$ of size $|U| \leq \beta k$ we have $e_{G_p}(U) \leq c_4 |U|$.
		\end{claim}
		
		It follows from (\ref{e:excess}) and Claim \ref{c:fewedges} that whp $G_{p}[C_0]$ is $(c_3, c_4, \beta)$-locally sparse. 
		
		We shall show that the effect of vertices of large degree on all these estimates is small, so that we can assume that $G_p[C_0]$ has bounded maximum degree. To do this we use a result from \cite{K18}, which says that whp no small set of vertices meets too many edges.
		
		\begin{claim}[{\cite[Proposition 2]{K18}}]\label{c:touchesfewedges}
			If $\mu>0$ is sufficiently small and $f(\mu) = - \mu \log \mu$, then whp every set of at most $f(\mu)n$ vertices in $G_p$ touches at most $\mu n$ edges.
		\end{claim}
		
		Note that $f(\mu) \rightarrow 0$ as $\mu \rightarrow 0$. Let $Y$ be the $f(\mu)n$ vertices of highest degree in $G_p$. If Claim \ref{c:touchesfewedges} holds, then all vertices in $G_p \setminus Y$ have degree at most $2 \frac{\mu}{f(\mu)}$, since otherwise the vertices in $Y$ would meet more than
		\[
		\frac{1}{2} f(\mu) n 2 \frac{\mu}{f(\mu)} = \mu n
		\]
		many edges in $G_p$, contradicting the claim. Hence, whp in $G' = G_p \setminus Y$ the vertex set $C' = C_0 \setminus Y$ will span at least $c_3 |C_0| - \mu n = \left(1 + \frac{c_2 \varepsilon}{4}\right) |C_0| - \mu n$ many edges. Since $|C_0| \geq c_1 k \geq c_1 \nu n$, if $\mu(c_1,c_2,\nu)$ is sufficiently small, then there will be at least 
		\[
		\left(1+\frac{c_2 \varepsilon}{5}\right) |C_0| \geq \left(1+\frac{c_2 \varepsilon}{5}\right) |C_0 \setminus Y| =: c_3'|C_0 \setminus Y|
		\]
		many edges in  $G'$. Note that $c'_3 > c_4 > 1$. Furthermore, every set of at most $\beta k$ vertices in $G'$ is also a subset of $G_p$ and so has at most $c_4 |U|$ many edges.
		
		It follows that whp $G'$ is $(c'_3,c_4,\beta)$-locally sparse and its maximum degree is bounded above by $2 \mu / f(\mu)$. Hence, by Lemma \ref{t:locallysparseexpander}, whp $G'$ contains a linear sized (in $n$) expander with bounded maximum degree, and so by Lemma \ref{t:expandingminor}, whp $G'$ (and hence $G_p$) contains a complete minor of order $\Omega(\sqrt{n}) = \Omega(\sqrt{k})$.
	\end{proof}

	\smallskip
	It remains  to prove Claims \ref{c:edgesinlargecomp} and \ref{c:fewedges}.
	
	\begin{proof}[Proof of Claim \ref{c:edgesinlargecomp}]
		Let us say a component $C$ of $G_{p_1}$ is \emph{bad} if $|C| \geq c_1 k$ and $e_G(C) < c_2k |C|$. There are at most $\sum_{r \geq c_1 k} \binom{n}{r}$ possible vertex sets for bad components $C$. Furthermore, for each such set  $C$ with $|C|=r$, we have 
		\begin{align*}
		\bb{P}( C \text{ is a bad component}) &\leq \bb{P}(C \text{ is a component in } G_{p_1} | e_G(C) < c_2 k r)\\ &\leq \binom{c_2 k r}{r-1} p_1^{r-1}\\
		&\leq (2e c_2 k)^{r-1} \left( \frac{1+\frac{\varepsilon}{2}}{k} \right)^{r-1}\\
		&= \left(2e c_2 \left(1+ \frac{\varepsilon}{2}\right)\right)^{r-1}.
		\end{align*} 
		Hence, by the union bound, we have
		\begin{align*}
		\bb{P}(\text{there exists a bad }C ) &\leq \sum_{r =c_1 k}^n \binom{n}{r}  \left(2e c_2 \left(1+ \frac{\varepsilon}{2}\right)\right)^{r-1}\\
		&\leq \sum_{r = c_1 k}^n \left(\frac{en}{c_1 k}\right)^r \left(2e c_2 \left(1+ \frac{\varepsilon}{2}\right)\right)^{r-1}\\
		&= \sum_{r =\ c_1 k}^n \left( \frac{e}{c_1\nu}\right)^r\left(2e c_2 \left(1+ \frac{\varepsilon}{2}\right)\right)^{r-1} = o(1),
		\end{align*}
		as long as $c_2 = c_2(c_1,\varepsilon,\nu)$ is sufficiently small.
	\end{proof}
	
	\begin{proof}[Proof of Claim \ref{c:fewedges}]
		The proof goes via the union bound as above. We say a subset $U$ is \emph{bad} if $|U| \leq \beta k$ but $e_{G_p}(U) \geq c_4 |U|$. Then we have
		\begin{align*}
		\bb{P}( \text{there exists a bad } U) &\leq \sum_{r = 1}^{\beta k} \binom{n}{r} \binom{\binom{r}{2}}{c_4r}p^{c_4r}\\
		&\leq \sum_{r = 1}^{\beta k} \left( \frac{en}{r}\left( \frac{er (1+\varepsilon)}{k} \right)^{c_4} \right)^r \\
		&= \sum_{r = 1}^{\beta k} \left( e (e(1+ \varepsilon))^{c_4} \nu^{-1} \left( \frac{r}{k} \right)^{c_4-1} \right)^r\\
		&\leq \sum_{r = 1}^{\beta k} \left( e (e(1+ \varepsilon))^{c_4} \nu^{-1}  \beta^{c_4-1} \right)^r,
		\end{align*}
		which will be $o(1)$ as long as $\beta = \beta(c_4, \varepsilon,\nu)$ is sufficiently small.
	\end{proof}
	
	\subsection*{Acknowledgement} Part of this work was performed when the third author visited the Institute of Discrete Mathematics at TU Graz and this visit was financially supported by TU Graz within
	the Doctoral Program \lq\lq Discrete Mathematics\rq\rq. He is grateful to the Institute of Discrete Mathematics for its hospitality. We would also like to thank the reviewers for their suggestions and comments.

	
	\bibliographystyle{plain}
	\bibliography{minor}
	
\end{document}